\documentclass[12pt]{article}
\usepackage{amsmath,amssymb,amsfonts,amsthm,url}
\usepackage[margin=1in]{geometry}

\usepackage{graphicx}
\usepackage{subcaption}

\theoremstyle{plain}
\newtheorem{theorem}{Theorem}[section]
\newtheorem{conjecture}[theorem]{Conjecture}
\newtheorem{proposition}[theorem]{Proposition}

\newcommand{\Z}{\mathbb{Z}}

\title{Index 3 biembeddings of the complete graphs}
\author{Juvenal F.\ Barajas and Timothy Sun\\Department of Computer Science\\San Francisco State University}
\date{}

\begin{document}

\maketitle

\begin{abstract}
We show that the complete graphs on $24s+21$ vertices have decompositions into two edge-disjoint subgraphs, each of which triangulates an orientable surface. The special case where the two surfaces are homeomorphic solves a generalized Earth-Moon problem for that surface. Unlike previous constructions, these pairs of triangular embeddings are derived from index 3 current graphs.
\end{abstract}

\section{Introduction}

There are many graph parameters that generalize the notion of planarity. Perhaps the most well-known of such parameters is the \emph{genus} of the graph, which is the smallest value $g$ such that the graph has an embedding in $S_g$, the orientable surface of genus $g$. A less-studied parameter is the \emph{thickness} of a graph, which is the size of the smallest partition of the edges into planar subgraphs. A graph is said to be \emph{biembeddable} in surfaces $S$ and $S'$ if it can be decomposed into two edge-disjoint subgraphs, one of which embeds in $S$ and the other embeds in $S'$. When $S$ is homeomorphic to $S'$, we simply say that the graph is biembeddable in $S$. We consider a variant of both genus and thickness, the \emph{bigenus} of a graph $\beta(G)$, which is defined to be the smallest value $g$ such that the graph $G$ is biembeddable in $S_g$.

The \emph{Earth-Moon problem} is a longstanding open problem on the maximum possible chromatic number of a graph with thickness 2, or equivalently, bigenus 0. At present, it is known that this value is 9, 10, 11, or 12 (see \cite{Gethner-ToTheMoon}). The upper bound is derived from a standard coloring argument based on average degree, which Heawood \cite{Heawood-MapColour} also uses to color graphs embedded in arbitrary orientable surfaces. Heawood's conjecture that his upper bound is tight is now called the Map Color Theorem \cite{Ringel-MapColor}, proven by Ringel, Youngs, \emph{et al.} 

Jackson and Ringel \cite{JacksonRingel-Variations} conjecture a similar result for graphs \emph{bi}embeddable in higher-genus orientable surfaces. The maximum chromatic number over all graphs biembeddable in the surface $S_g$ is called the \emph{bichromatic number} of $S_g$ and is denoted by $\chi_2(S_g)$. The same coloring argument is used to prove the following Heawood-like inequality:

\begin{proposition}[Jackson and Ringel \cite{JacksonRingel-Variations}]
The bichromatic number of the orientable surface $S_g$, where $g \geq 1$, is at most
$$\chi_2(S_g) \leq \left\lfloor \frac{13+\sqrt{73+96g}}{2} \right\rfloor.$$
\label{prop-color}
\end{proposition}

\begin{conjecture}[Jackson and Ringel \cite{JacksonRingel-Variations}]
For all $g \geq 1$, the bound in Proposition \ref{prop-color} is tight.
\label{conj-upper}
\end{conjecture}

Just like the Map Color Theorem, this generalization of the Earth-Moon problem hardly resembles the original problem on the sphere: for all other surfaces, one might expect that the upper bound is always matched by a biembedding of a complete graph on the same number of vertices. Conjecture \ref{conj-upper} thus has a stronger ``graph-centric'' formulation in terms of bigenus:

\begin{proposition}[Cabaniss and Jackson \cite{CabanissJackson-Biembeddings}]
The bigenus of the complete graph $K_n$ is at least
$$\beta(K_n) \geq \left\lceil \frac{n^2-13n+24}{24} \right \rceil.$$
\label{prop-genus}
\end{proposition}

\begin{conjecture}[Cabaniss and Jackson \cite{CabanissJackson-Biembeddings}]
For all $n \geq 11$, $$\beta(K_n) = \left\lceil \frac{n^2-13n+24}{24} \right \rceil.$$
\label{conj-lower}
\end{conjecture}

The bigenus of the complete graph $K_n$ can equal exactly $n^2-13n+24/24$ only when both embeddings of the biembedding are triangular. These so-called \emph{triangular biembeddings} are only possible when $n \equiv 0, 13, 16, 21 \pmod{24}$, otherwise the expression is not an integer. With the exception of some small cases ($\beta(K_n)$ is known for all $n \leq 14$ \cite{Ringel-Farbungsprobleme, BattleHararyKodama, Tutte-K9, Ringel-Toroidal, Beineke-TwoTorus}), all other known constructions of minimum genus biembeddings of $K_n$ have been triangular biembeddings. The second author \cite{Sun-Bigenus} found triangular embeddings of self-complementary graphs on $16$, $21$, and $24$ vertices through computer search. One of the aforementioned residues, $n \equiv 13 \pmod{24}$, has been solved using \emph{current graphs}, a covering space construction that has proven to be effective for finding triangular embeddings of dense graphs. The application of current graphs to biembeddings was initiated by Anderson and White \cite{AndersonWhite}, who found a pair of current graphs that produce a triangular biembedding of $K_{37}$. Cabaniss and Jackson \cite{CabanissJackson-Biembeddings} then solved the bigenus of $K_{61}$ and $K_{85}$. Finally, the second author \cite{Sun-Bigenus} completed this line of work by finding an infinite family of current graphs that produce triangular biembeddings of the complete graphs on $n = 24s+13$ vertices, for all $s \geq 1$. 

The aforementioned current graphs are all of \emph{index 1}, i.e., they are all 1-face embeddings. We solve another one of the residues by constructing triangular biembeddings of the complete graphs $K_{24s+21}$, for all $s \geq 0$, using \emph{index 3} current graphs. 

\section{Graph embeddings}

We assume prior knowledge of topological graph theory and the theory of current graphs. For background on these topics, see Gross and Tucker \cite{GrossTucker} and Ringel \cite{Ringel-MapColor}. In particular, Section 9 of Ringel \cite{Ringel-MapColor} describes current graph constructions similar to the ones we will present here. For more information on the thickness parameter and its variants, see Beineke \cite{Beineke-Survey}. 

A \emph{cellular embedding} of a graph $G = (V,E)$ in the surface $S_g$ is an injective mapping $\phi\colon G \to S_g$, where the components of $S_g \setminus \phi(G)$ are open disks. We call these disks \emph{faces}. In this paper, all graph embeddings are cellular and in orientable surfaces. If the set of faces is denoted by $F(\phi)$, then its size is determined by the \emph{Euler polyhedral equation}
$$|V|-|E|+|F(\phi)| = 2-2g.$$ 
When $G$ is simple, the Euler polyhedral equation implies a well-known inequality on the number of edges in $G$:
\begin{proposition}
If $G = (V,E)$ is a simple graph embedded in the orientable surface $S_g$, then
$$|E| \leq 3|V|-6+6g,$$
with equality if and only if the embedding is triangular. 
\label{prop-singlebound}
\end{proposition}
For biembeddings, a graph can have twice as many edges, and one can use this inequality to prove Propositions \ref{prop-color} and \ref{prop-genus}. 

To describe a cellular embedding combinatorially, each edge $e \in E$ induces two arcs $e^+$ and $e^-$ with the same endpoints, each representing the two different directions in which $e$ can be traversed. The set of such arcs is denoted $E^+$. A \emph{rotation} of a vertex is a cyclic permutation of the arcs leaving that vertex, and a \emph{rotation system} of a graph is an assignment of a rotation to each vertex. When a graph is simple, it is sufficient to describe a rotation as a cyclic permutation of the vertex's neighbors. The Heffter-Edmonds principle states that rotation systems are in one-to-one correspondence with cellular embeddings in orientable surfaces (see Section 3.2 of Gross and Tucker \cite{GrossTucker}). From a rotation system, a cellular embedding can be found through face-tracing, where each face-boundary walk corresponds to a cyclic sequence of arcs $(e_1^\pm, e_2^\pm, \dotsc, e_i^\pm)$. 

\section{Current graphs}

A \emph{current graph} is an arc-labeled, embedded graph where the arc-labeling $\alpha: E^+ \to \Z_n \setminus \{0\}$ satisfies $\alpha(e^+) = -\alpha(e^-)$ for each edge $e$. We call $\Z_n$ the \emph{current group} and the arc labels \emph{currents}. The \emph{index} of a current graph is the number of faces in the embedding. Our current graphs are of index 3, and its face-boundary walks, which we call \emph{circuits}, are labeled $[0]$, $[1]$, and $[2]$. Given a circuit, the \emph{log} of the circuit replaces each arc with its current. We require that our current graphs satisfy a standard set of properties:

\begin{enumerate}
\item[(E1)] The current graph has index 3. 
\item[(E2)] Each vertex has degree 3 and satisfies KCL.
\item[(E3)] Each nonzero element of the current group $\Z_{3m}$ appears at most once in the log of each circuit. 
\item[(E4)] If circuit $[a]$ traverses arc $e^+$ and circuit $[b]$ traverses arc $e^-$, then $\alpha(e^+) \equiv b-a \pmod{3}$. 
\end{enumerate}

The \emph{derived embedding} of a current graph satisfying the above properties is constructed in the following way: the vertex set is the current group $\Z_{3m}$, and the rotation at any vertex $i \in \Z_{3m}$ (and hence its set of neighbors) is found by taking the log of circuit $[i \bmod{3}]$ and adding $i$ (modulo $\Z_{3m}$) to each element. A vertex $i$ is called a $[k]$-vertex if $i \bmod{3} = k$, i.e., it is a vertex whose rotation is determined by circuit $[k]$. 

Since every vertex has degree 3 and satisfies KCL, the derived embedding is triangular. Its genus thus has a simple formula:

\begin{proposition}
Given an index 3 current graph, if the number of vertices is $v$, the current group is $\Z_{3m}$, and the derived embedding is connected, then its genus is $(v-6)m/4+1$.
\label{prop-cgenus}
\end{proposition}
\begin{proof}
Since there are three circuits and every vertex has degree 3, the average length of a circuit, and hence the average degree of the graph, is $v$. The above formula results from substituting $E = 3mv/2$ and $V = 3m$ into Proposition \ref{prop-singlebound}.
\end{proof}

Our current graphs come in pairs, and each pair satisfies two additional properties:

\begin{itemize}
\item[(E5)] For each $k = 0, 1, 2$, each nonzero element of $\Z_{3m}$ appears in the log of circuit $[k]$ in exactly one of the two current graphs. 
\item[(E6)] Both current graphs have the same number of vertices.
\end{itemize}

When these properties are satisfied, each possible edge between distinct vertices appears in exactly one of the two derived embeddings and by Proposition \ref{prop-cgenus}, the derived embeddings are on surfaces of the same genus. Consequently, we have a triangular biembedding of the complete graph $K_{3m}$. 

\begin{figure}[!t]
\centering
\includegraphics[scale=0.8]{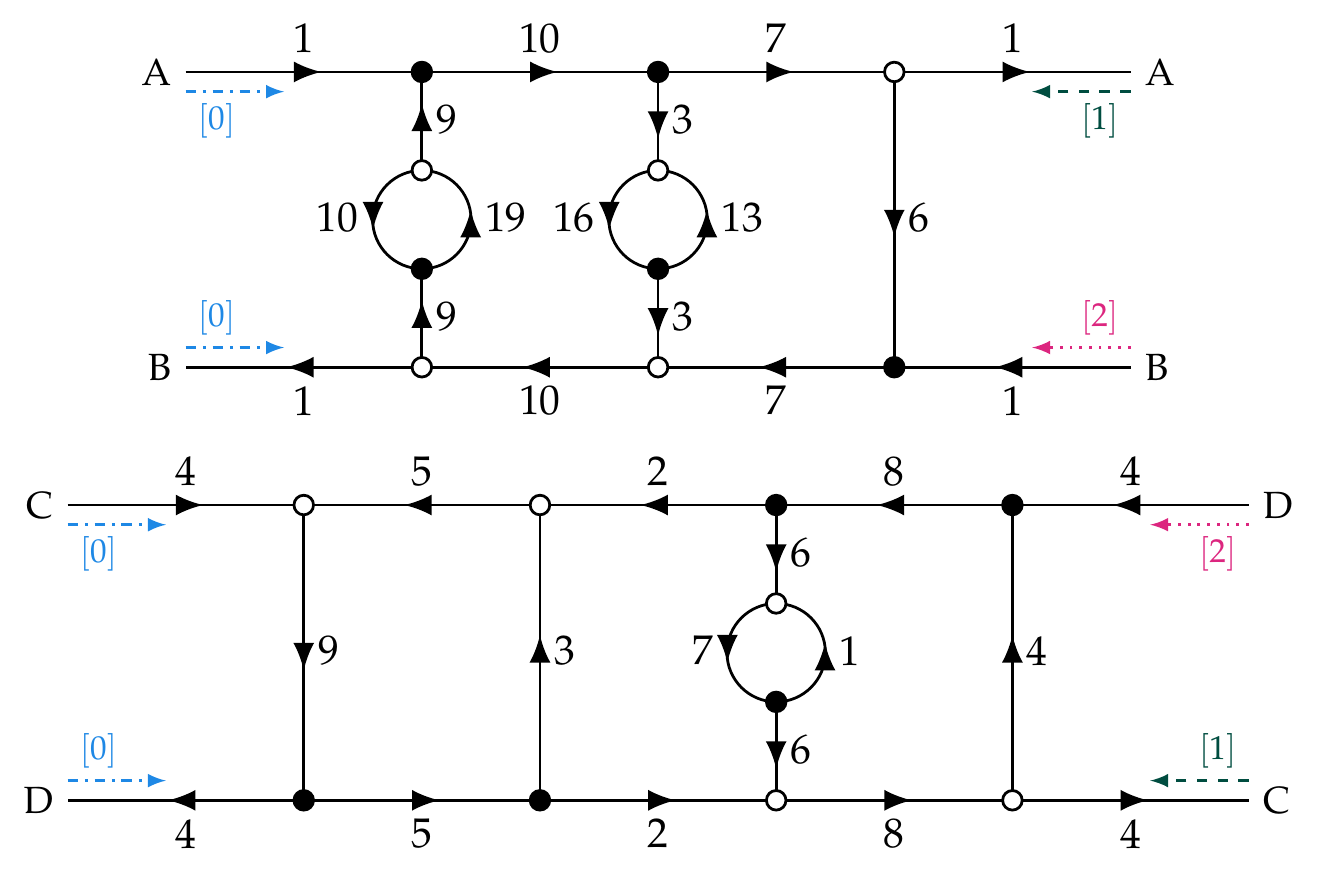}
\caption{A pair of current graphs over $\mathbb{Z}_{21}$.}
\label{fig-s0}
\end{figure}

The two current graphs in Figure \ref{fig-s0} satisfy properties (E1)--(E6). Hence, their derived embeddings form a triangular biembedding of $K_{21}$. These current graphs contain frequently used elements in index 3 constructions that were first described in detail by Youngs \cite{Youngs-3569}. The underlying graphs are (circular or M\"obius) \emph{ladders} containing \emph{rungs}. The rungs come in two varieties: \emph{simple} rungs that are just vertical edges, and \emph{ring-shaped} rungs, which have two more vertices connected by two parallel edges. 

\section{The main construction}

\begin{figure}[!t]
\centering
\includegraphics[width=\textwidth]{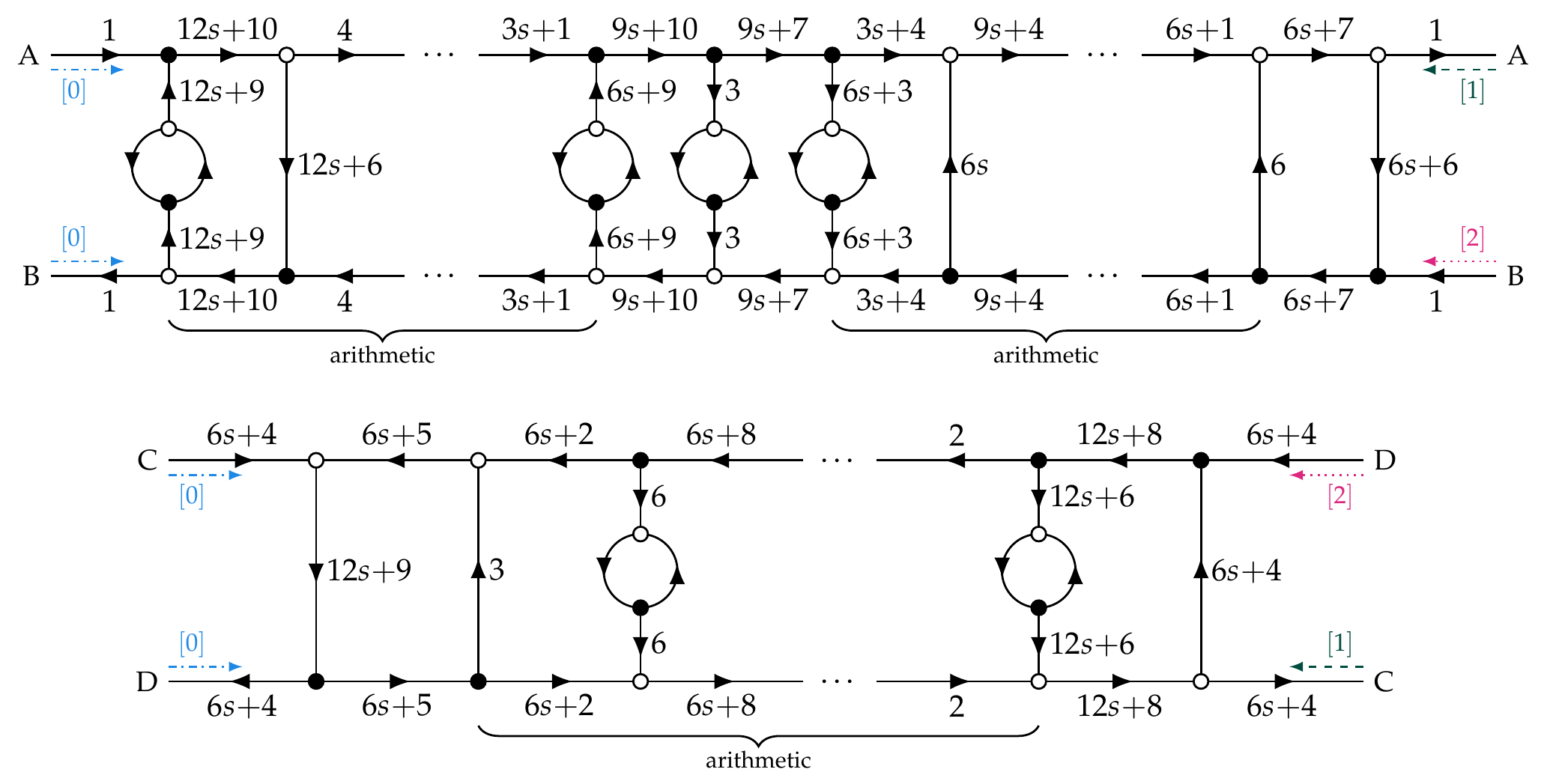}
\caption{Pairs of current graphs for all $s \geq 0$ with current group $\mathbb{Z}_{24s+21}$.}
\label{fig-general}
\end{figure}

\begin{figure}[!t]
\centering
\includegraphics[scale=0.8]{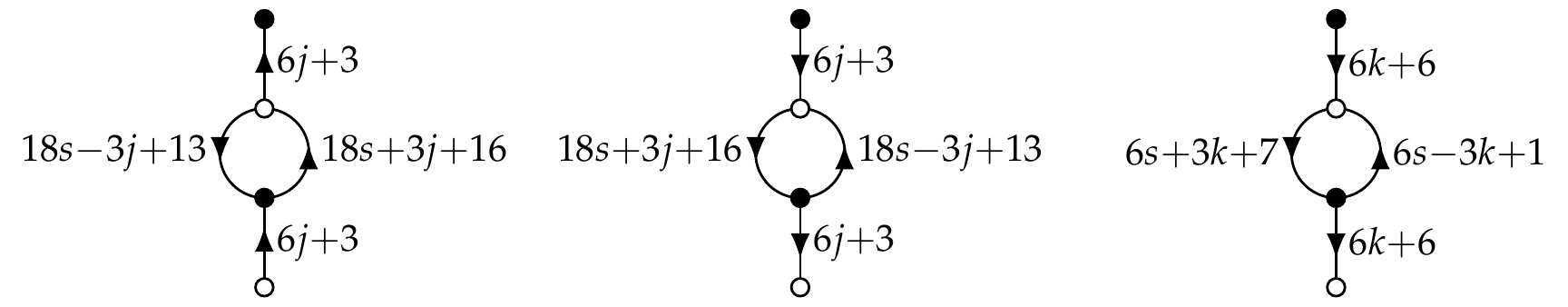}
\caption{Current assignments on circular arcs.}
\label{fig-rings}
\end{figure}

The current graphs in Figure \ref{fig-s0} constitute the smallest instance of an infinite family:

\begin{theorem}
The complete graph $K_{24s+21}$ has a triangular biembedding for all $s \geq 0$.
\label{thm-main}
\end{theorem}
\begin{proof}
The current graphs described in Figure \ref{fig-general} satisfy properties (E1)--(E6) and thus generate triangular biembeddings of the complete graphs $K_{24s+21}$, for all $s \geq 0$. The sections labeled ``arithmetic'' describe part of the ladder where:
\begin{itemize}
\item the rungs alternate between simple and ring-shaped,
\item the vertical arcs alternate in direction, and
\item the currents on those vertical arcs form an arithmetic sequence with step size 3.
\end{itemize}
In the interest of space, the labels on the circular arcs are given separately in Figure \ref{fig-rings}, where the variables have the ranges $j = 0, \dotsc, 2s+1$ and $k = 0, \dotsc, 2s$. To check that the derived embeddings partition the edges of $K_{24s+21}$, we categorize the edges based on their incident circuits. The horizontal edges are where circuit $[0]$ meets with either circuit $[1]$ or $[2]$; the simple rungs are where circuit $[0]$ meets with itself; the vertical edges of ring-shaped rungs are where circuits $[1]$ and $[2]$ meet with themselves; and the circular arcs are where circuits $[1]$ and $[2]$ meet. One can use this information to check that property (E5) is satisfied.

In both current graphs, there is at least one edge incident with circuits $[0]$ and $[1]$, and at least one edge incident with circuits $[0]$ and $[2]$. Because of the presence of an arc with current $3$, the derived embeddings of the first and second current graphs have a cycle passing through all the $[1]$-vertices and $[0]$-vertices, respectively. These two properties imply that the derived embeddings are connected. 
\end{proof}

\section{Biembeddings on different surfaces}

Rearranging parts of the above infinite families of current graphs results in biembeddings into two surfaces of different genus. Cabaniss and Jackson \cite{CabanissJackson-Biembeddings} say that a graph is \emph{$(g,h)$-biembeddable} if it has an edge decomposition into two subgraphs, one of which is embeddable in the surface $S_g$, and the other in $S_h$.

For each pair of current graphs in our main construction, each multiple of 3 corresponds to two rungs: in one graph, it appears as a current on a simple rung, and in the other graph, it appears twice on the vertical arcs of a ring-shaped rung. These two rungs can be swapped (possibly with some changes in arc directions) while preserving properties (E2)--(E5). Such exchanges have appeared in other constructions of index 3 current graphs (see, e.g., \cite{JungermanRingel-Minimal,Sun-Minimum}), except in those cases, they were rungs in the same current graph. In our situation, two pairs of rungs need to be swapped at the same time to ensure property (E1), that the indices of both current graphs stay at 3. Property (E6) is violated intentionally to get derived embeddings on different surfaces. 

\begin{figure}[!t]
\centering
\includegraphics[scale=0.8]{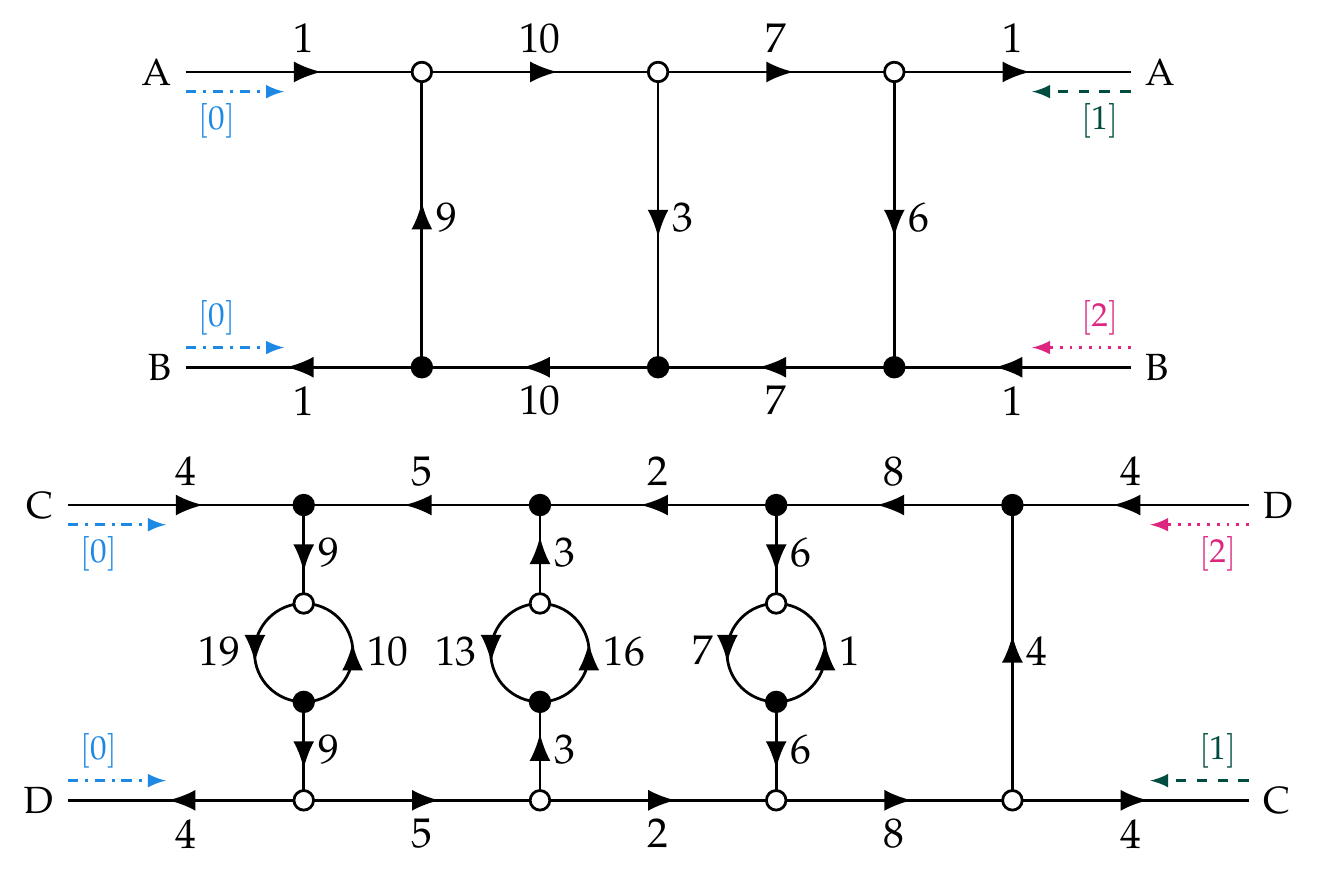}
\caption{$K_{21}$ is $(1,15)$-biembeddable.}
\label{fig-flipped}
\end{figure}

\begin{theorem}
The complete graph $K_{24s+21}$ is $$(b(s)-(8s+7)k, b(s)+(8s+7)k)\text{-biembeddable},$$ where $b(s) = \beta(K_{24s+21}) = 24s^2+29s+8$ and $k = 0, \dotsc, s+1$. 
\label{thm-swap}
\end{theorem}
\begin{proof}
Switching two normal rungs with two ring-shaped rungs changes the total number of vertices in both current graphs by $4$. From Proposition \ref{prop-cgenus}, the genus must increase or decrease by $8s+7$. The first current graph has $2s+2$ ring-shaped rungs (the second current graph has one fewer), so up to $s+1$ pairs of rungs can be exchanged. Finally, the connectivity argument at the end of the proof of Theorem \ref{thm-main} is still valid even if the rungs with current 3 are swapped.
\end{proof}

Figure \ref{fig-flipped} shows a swap on the two current graphs that originally appeared in Figure \ref{fig-s0}. Plugging in $s = 0$ and $k = 1$ into Theorem \ref{thm-swap} shows that the derived embeddings of the graphs are on the torus and the genus 15 surface. 

\bibliographystyle{alpha}
\bibliography{biblio}

\newpage

\end{document}